\documentclass[12pt,reqno]{amsart}
\usepackage{latexsym}
\usepackage{amssymb}
\usepackage{mathrsfs}
\usepackage{amsmath}
\usepackage{fancybox,color}
\usepackage{enumerate}
\usepackage[latin1]{inputenc}
\usepackage{soul} 
\usepackage[colorlinks=true, linkcolor=blue, citecolor=blue]{hyperref}

\usepackage[left=2.3cm,top=2.1cm,right=2.3cm]{geometry}

\def\1{\raisebox{2pt}{\rm{$\chi$}}}

\newtheorem{theorem}{Theorem}[section]

\newtheorem{lemma}[theorem]{Lemma}
\newtheorem{proposition}[theorem]{Proposition}

\theoremstyle{definition}
\newtheorem{remark}[theorem]{Remark}

\newcommand{\R}{{\mathbb R}}

\newcommand{\N}{{\mathbb N}}
\newcommand{\Z}{{\mathbb Z}}

\newcommand{\Ha}{{\mathcal H}}

\newcommand\cp{\operatorname{cap}}
\newcommand\diam{\operatorname{diam}}

\newcommand{\eps}{{\varepsilon}}
\def\1{\raisebox{2pt}{\rm{$\chi$}}}

\newcommand{\Lip}{\operatorname{Lip}}

\def\vint_#1{\mathchoice%
        {\mathop{\kern 0.2em\vrule width 0.6em height 0.69678ex depth -0.58065ex
                \kern -0.8em \intop}\nolimits_{\kern -0.4em#1}}%
        {\mathop{\kern 0.1em\vrule width 0.5em height 0.69678ex depth -0.60387ex
                \kern -0.6em \intop}\nolimits_{#1}}%
        {\mathop{\kern 0.1em\vrule width 0.5em height 0.69678ex
            depth -0.60387ex
                \kern -0.6em \intop}\nolimits_{#1}}%
        {\mathop{\kern 0.1em\vrule width 0.5em height 0.69678ex depth -0.60387ex
                \kern -0.6em \intop}\nolimits_{#1}}}
\def\vintslides_#1{\mathchoice%
        {\mathop{\kern 0.1em\vrule width 0.5em height 0.697ex depth -0.581ex
                \kern -0.6em \intop}\nolimits_{\kern -0.4em#1}}%
        {\mathop{\kern 0.1em\vrule width 0.3em height 0.697ex depth -0.604ex
                \kern -0.4em \intop}\nolimits_{#1}}%
        {\mathop{\kern 0.1em\vrule width 0.3em height 0.697ex depth -0.604ex
                \kern -0.4em \intop}\nolimits_{#1}}%
        {\mathop{\kern 0.1em\vrule width 0.3em height 0.697ex depth -0.604ex
                \kern -0.4em \intop}\nolimits_{#1}}}

\newcommand{\intav}{\vint}
\newcommand{\aveint}[2]{\mathchoice%
        {\mathop{\kern 0.2em\vrule width 0.6em height 0.69678ex depth -0.58065ex
                \kern -0.8em \intop}\nolimits_{\kern -0.45em#1}^{#2}}%
        {\mathop{\kern 0.1em\vrule width 0.5em height 0.69678ex depth -0.60387ex
                \kern -0.6em \intop}\nolimits_{#1}^{#2}}%
        {\mathop{\kern 0.1em\vrule width 0.5em height 0.69678ex depth -0.60387ex
                \kern -0.6em \intop}\nolimits_{#1}^{#2}}%
        {\mathop{\kern 0.1em\vrule width 0.5em height 0.69678ex depth -0.60387ex
                \kern -0.6em \intop}\nolimits_{#1}^{#2}}}

\newcommand{\ol}{\overline}
\newcommand{\wtilde}{\widetilde}

\newcommand{\dist}{\operatorname{dist}}

\title[Self-improvement of uniform fatness]{Self-improvement of uniform fatness revisited}

\author[J.\! Lehrb\"ack]{Juha Lehrb\"ack}   
\address[J.L.]{University of Jyvaskyla, Department of Mathematics and Statistics, P.O. Box 35, FI-40014 University of Jyvaskyla, Finland}
\email{juha.lehrback@jyu.fi}

\author[H.\! Tuominen]{Heli Tuominen}   
\address[H.T.]{University of Jyvaskyla, Department of Mathematics and Statistics, P.O. Box 35, FI-40014 University of Jyvaskyla, Finland} \email{heli.m.tuominen@jyu.fi}

\author[A.V.\! V\"ah\"akangas]{Antti V. V\"ah\"akangas}
\address[A.V.V.]{University of Jyvaskyla, Department of Mathematics and Statistics, P.O. Box 35, FI-40014 University of Jyvaskyla, Finland} \email{antti.vahakangas@iki.fi}

\begin{document}



\begin{abstract}
We give a new proof for the self-improvement of uniform
$p$-fatness in the setting of general metric spaces. 
Our proof is based on rather standard methods
of geometric analysis, and in particular the proof avoids
the use of deep results from potential theory and analysis 
on metric spaces that have been indispensable in the previous
proofs of the self-improvement. A key ingredient in the proof
is a self-improvement property for local Hardy inequalities.
\end{abstract}

\maketitle

\section{Introduction}

Self-improvement is among the most profound and beautiful phenomena
in mathematical analysis, and a source of important tools in the proofs of several deep 
and perhaps even surprising results.
Important examples of concepts enjoying self-improvement include reverse H\"older inequalities, 
Muckenhoupt's $A_p$ classes of weights, Poincar\'e inequalities,
and the main topics of this paper: Hardy inequalities and 
uniform $p$-fatness related to the variational $p$-capacity.

That a uniformly $p$-fat set $E$, for $1<p<\infty$, is actually uniformly $q$-fat
for some $1\le q<p$ as well, was first proven by Lewis~\cite{MR946438} in the 
Euclidean case $E\subset\R^n$.
In fact, Lewis studied more general $(\alpha,p)$-fatness conditions  
related to Riesz capacities,
but when $\alpha=1$ his setting is equivalent to that of the variational $p$-capacity.
Another proof for the self-improvement of uniform $p$-fatness in (weighted) $\R^n$ was given by
Mikkonen~\cite{MR1386213}, and in~\cite{MR1869615} Bj\"orn, MacManus and Shanmugalingam generalized
the self-improvement to more general metric spaces,  essentially proving the 
following theorem (although in~\cite{MR1869615} the
assumptions on the space $X$ were slightly stronger).

\begin{theorem}\label{t.self-impr}
Let $1< p<\infty$ and let $X$ be a complete 
metric measure space equipped with a doubling measure $\mu$ and 
supporting a $(1,p)$-Poincar\'e inequality.
Assume that $E\subset X$ is a  uniformly $p$-fat closed set.
Then there exists $1<q<p$ such that $E$ is also uniformly $q$-fat
(quantitatively).
\end{theorem}

The proofs of the versions of Theorem~\ref{t.self-impr} 
in~\cite{MR1869615,MR946438,MR1386213} utilize deep results from  linear and
non-linear potential theory, and moreover the proof in~\cite{MR1869615} 
is based on the impressive theory of differential structures on metric spaces, 
established by Cheeger in~\cite{MR1708448}. 

In this paper, we use a different approach
and establish a new proof for Theorem~\ref{t.self-impr} 
with the help of 
local Hardy inequalities and their self-improvement properties. 
Our proof is completely new also in $\R^n$, where all previously known proofs 
have been based on the ideas either in~\cite{MR946438} or in~\cite{MR1386213}. In addition, 
it turns out that with our approach it is possible
to obtain the following generalization of Theorem~\ref{t.self-impr} to
a non-complete space $X$, where Cheeger's theory is not available.

\begin{theorem}\label{t.self-impr-non-compl_Intro}
Let $1<p_0<p<\infty$ and let $X$ be a  
metric measure space equipped with a doubling measure $\mu$ and 
supporting a $(p,p_0)$-Poincar\'e inequality. 
Assume that $E\subset X$ is a  uniformly $p$-fat closed set
and that $E\cap \ol{B(w,r)}$ is compact for all $w\in E$ and all $r>0$.
Then there exists $p_0< q<p$ such that $E$ is also uniformly $q$-fat
(quantitatively).
\end{theorem}

We will also formulate a slightly stronger
version of Theorem~\ref{t.self-impr-non-compl_Intro} later in Theorem~\ref{t.self-impr-non-compl}.
Recall that if $X$ is as in Theorem~\ref{t.self-impr} (i.e.,\ complete, equipped with a doubling measure
and supporting a $(1,p)$-Poincar\'e inequality), then a $(p,p_0)$-Poincar\'e inequality as in
Theorem~\ref{t.self-impr-non-compl_Intro} follows 
from the well-known self-improvement properties of Poincar\'e inequalities; see
Section~\ref{s.poinc} for more discussion. 

It should perhaps be noted here that 
we define the 
variational $p$-capacity
using Lipschitz test functions (the precise definition is given in
Section~\ref{s.cap}). If $X$ is complete, this definition
agrees with the definition using Newtonian (or Sobolev)
test functions, but in a non-complete space
the resulting capacities can be different.

Let us turn to an outline of the ideas behind the proofs of 
Theorems~\ref{t.self-impr} and~\ref{t.self-impr-non-compl_Intro}.
In~\cite{MR2854110} (see also~\cite[Theorem~3.3]{MR2723821}) it was shown (essentially) that
if $X$ is as in Theorem~\ref{t.self-impr}, then a closed set $E\subset X$ is uniformly
$p$-fat if and only if there is $C>0$ such that
the \emph{boundary $p$-Poincar\'e inequality}
\begin{equation}\label{e.b-poinc}
\int_{B(w,r)} \lvert u\rvert^{p}\,d\mu
\le Cr ^{p}\int_{B(w,\tau r)}  g^{p}\,d\mu\,
\end{equation}
holds for all $w\in E$ and all $r>0$, whenever
$u$ is a Lipschitz function in $X$ such that $u=0$ in $E$
and $g$ is a ($p$-weak) upper gradient of $u$
(in $\R^n$ one can always take $g=|\nabla u|$).
Hence to obtain the self-improvement of uniform $p$-fatness, it would suffice to
prove the self-improvement directly to inequality~\eqref{e.b-poinc}; 
this was actually mentioned in~\cite[p.~718]{MR2854110} as 
a possible and interesting approach to self-improvement.

We will not give a direct proof 
for the self-improvement 
of~\eqref{e.b-poinc}, but we show in Theorem~\ref{t.main} 
that if $X$ is as in Theorem~\ref{t.self-impr-non-compl_Intro} 
(in particular not necessarily complete) and
$E\subset X$ is uniformly $p$-fat, then
there exist $\eps>0$ and $C>0$ such that
the following $(p-\eps)$-version of~\eqref{e.b-poinc} holds 
for all $w\in E$ and all $r>0$, whenever
$u$ is a Lipschitz function in $X$ such that $u=0$ in $E$:
\begin{equation}\label{e.ultimate_intro}
\int_{B(w,r)} \lvert u\rvert^{p-\varepsilon}\,d\mu
\le Cr ^{p-\varepsilon}\int_{B(w,\tau r)}  \Lip(u,\cdot)^{p-\varepsilon}\,d\mu\,.
\end{equation}
Here $\Lip(u,x)$ is the upper pointwise Lipschitz constant of $u$ at $x\in X$.
We remark that in $\R^n$ inequality~\eqref{e.ultimate_intro} can be obtained directly
with $|\nabla u|$ instead of $\Lip(u,\cdot)$ on the right-hand side.
More generally, if the space $X$ is complete, then $\Lip(u,\cdot)$ 
is actually known to be a minimal weak
upper gradient of $u$ by the results of Cheeger~\cite{MR1708448}.
Hence we can connect from inequality~\eqref{e.ultimate_intro}
back to uniform fatness, and now indeed to the better $(p-\eps)$-uniform fatness,
thus proving Theorem~\ref{t.self-impr}.
In~\cite{MR2854110} this connection was
established with the help of the so-called pointwise Hardy inequalities,  
but, for the sake of completeness, we show in Section~\ref{s.self-impr} 
how Theorem~\ref{t.self-impr} follows directly from 
the validity of~\eqref{e.ultimate_intro}. 
In particular, this way we avoid the use of pointwise Hardy inequalities
in our proofs of Theorems~\ref{t.self-impr} 
and~\ref{t.self-impr-non-compl_Intro}, although it should be noted that
our general approach has been partially suggested and motivated by
these pointwise inequalities. 

In a non-complete space $X$ the validity of~\eqref{e.ultimate_intro} does not immediately
yield the uniform $(p-\eps)$-fatness of $E$. Nevertheless, using 
as an additional tool the
connection between uniform fatness and density conditions for suitable Hausdorff contents,
we show in Section~\ref{s.Hausd} how the improved boundary Poincar\'e inequality~\eqref{e.ultimate_intro}
can be used to prove also Theorem~\ref{t.self-impr-non-compl_Intro}.

It is the proof of~\eqref{e.ultimate_intro} (assuming uniform $p$-fatness) 
that constitutes the main challenge in our
proofs of Theorems~\ref{t.self-impr} and~\ref{t.self-impr-non-compl_Intro}. 
In fact, we will establish~\eqref{e.ultimate_intro} via
a self-improvement property of suitable local Hardy inequalities.
Recall that one of the consequences of the self-improvement of uniform fatness, 
noted in each of~\cite{MR1869615,MR946438,MR1386213},
is the validity of a $p$-Hardy inequality in the complement of 
a uniformly $p$-fat set $E\subset X$.
However, using a method originating from Wannebo~\cite{MR1010807}
(see also~\cite[Section~5]{MR2854110}), it is also possible to
prove such a $p$-Hardy inequality 
without using the self-improvement of uniform fatness.
We use an adaptation of this latter method
together with a novel `local absorbtion argument' (Lemma~\ref{l.absorb}),
and prove in the end of Section~\ref{s.wannebo} 
the following local $p$-Hardy inequality when $E\subset X$ is uniformly 
$p$-fat.

\begin{theorem}\label{t.Wannebo_intro}
Let $1< p<\infty$ and let $X$ be a 
metric measure space equipped with a doubling measure $\mu$ and 
supporting a $(1,p)$-Poincar\'e inequality.
Assume that $E\subset X$ is a  uniformly $p$-fat closed set.
Then there exists a constant $C>0$ such that the
local $p$-Hardy inequality
\begin{equation}\label{e.int_intro}
\int_{B\setminus E}\bigg(\frac{\lvert u\rvert}{d_{E}}\bigg)^p\,d\mu 
\le C\int_{32\tau^2 B} g^p\,d\mu
\end{equation}
holds whenever $u$ is a Lipschitz function in $X$ such that $u=0$ in $E$,
$g$ is a $p$-weak upper gradient of $u$, and $B=B(w,r)$ is a ball with $w\in E$ and $0<r< (1/32) \diam(X)$. 
\end{theorem}

Above we have abbreviated $d_{E}(x)=\dist(x,E)$. 
Notice in particular that we do not 
need to assume in Theorem~\ref{t.Wannebo_intro} that the space $X$ is complete.

The next step towards~\eqref{e.ultimate_intro} 
is a self-improvement property for 
 local $p$-Hardy inequalities~\eqref{e.int_intro}. 
Here we need the assumption that $X$ supports
a $(p,p_0)$-Poincar\'e inequality for some $1<p_0<p$
(actually, it suffices to assume that $X$ supports $(q,q)$-Poincar\'e inequalities 
for all $p_0\le q\le p$, with uniform constants). 
In this case there exists $\eps>0$ such that a version of inequality~\eqref{e.int_intro}
holds with the exponent $p-\eps$, but now with the $p$-weak
upper gradient $g$ on the right-hand side of~\eqref{e.int_intro}
replaced with the upper pointwise Lipschitz constant $\Lip(u,\cdot)$;
see Proposition~\ref{p.improved}.
The proof of this self-improvement for local Hardy inequalities is based on ideas used by
Koskela and Zhong~\cite{MR1948106} in connection with the self-improvement
of usual $p$-Hardy inequalities; the ideas in~\cite{MR1948106} were, in turn,
inspired by the work of Lewis~\cite{MR1239922}.
Again the absorbtion Lemma~\ref{l.absorb} is
needed to obtain the local inequalities. 
The $(p-\eps)$-version of the local Hardy inequality now easily yields
the $(p-\eps)$-version of the boundary Poincar\'e inequality~\eqref{e.ultimate_intro},
see Section~\ref{s.self-impr}, concluding the proofs of Theorems~\ref{t.self-impr} 
and~\ref{t.self-impr-non-compl_Intro}.

Admittedly, the proofs 
that were outlined above are somewhat lengthy and
in many places still quite technical and delicate in the level of details, 
but one could argue that our general approach is nevertheless based on rather 
`elementary' (or `standard') tools.
In particular, we do not need any sophisticated prerequisites 
concerning potential theory and we can also avoid completely the use of 
Cheeger's deep theory---or, if this theory is used to give a more direct proof to
Theorem~\ref{t.self-impr}, the use is very explicit and localized; cf.\ 
the proof of Theorem~\ref{t.self-impr} at the end of Section~\ref{s.self-impr}.
In this sense we believe that our proof of Theorem~\ref{t.self-impr} is 
more transparent (also in $\R^n$)
than its predecessors and thus hopefully easier to adapt to further problems, for instance 
in connection with weighted capacities
or capacities of fractional order smoothness.

\section{Preliminaries}\label{s.prelim}

\subsection{Metric spaces}

We assume throughout the paper that 
$X=(X,d,\mu)$ is a metric measure space equipped with a metric $d$ and a 
positive complete Borel
measure $\mu$ such that $0<\mu(B(x,r))<\infty$
for all balls  $B=B(x,r)=\{y\in X\,:\, d(y,x)<r\}$. 
As in \cite[p.~2]{MR2867756},
we extend $\mu$ as a Borel regular (outer) measure on $X$.
In particular, the space $X$ is separable.
Let us emphasize that 
we do not, in general, require $X$ to be complete. 
If completeness is needed somewhere in the paper, we will
mention this explicitly.

We also assume that the measure $\mu$ is {\em doubling},
meaning that there is a constant $C_D\ge 1$, called
the {\em doubling constant of $\mu$}, such that
\begin{equation}\label{e.doubling}
\mu(2B) \le C_D\, \mu(B)
\end{equation}
for all balls $B=B(x,r)$ of $X$. 
Here we use for $0<t<\infty$ the notation $tB=B(x,tr)$. 
When $A\subset X$, we let $\ol A$ denote the closure of $A$,
and hence $\ol B$ always refers to the closure of the ball $B$, 
not to the corresponding closed ball.

Let $A\subset X$.
A function $u\colon A\to \R$ is said to be \emph{($L$-)Lipschitz}, for $0\le L<\infty$, if
\[
|u(x)-u(y)|\le L d(x,y)\qquad \text{ for all } x,y\in A\,.
\]
If $u:A\to \R$ is an $L$-Lipschitz function, then the classical McShane extension
\begin{equation}\label{McShane}
\tilde u(x)=\inf_{y\in A} \{ u(y) + L d(x,y)\}\,,\qquad x\in X\,,
\end{equation}
defines an $L$-Lipschitz function $\tilde u:X\to \R$
which satisfies
$\tilde u|_A = u$.
The set of all Lipschitz functions $u\colon A\to\R$
is denoted by $\Lip(A)$, and
\[
\Lip_0(A) = \{ u\in  \Lip(X) \,:\, u=0 \text{ in } X\setminus A\}\,.
\]

\subsection{(Weak) upper gradients}
By a {\em curve} we mean a nonconstant, rectifiable, continuous
mapping from a compact interval to $X$.
We say that a Borel function $g\ge 0$ on $X$ is an 
{\em upper gradient} of an extended real-valued function $u$ on
$X$, if for all curves $\gamma$ joining 
arbitrary points $x$ and $y$ in $X$ we have
\begin{equation}\label{e.modulus}
\lvert u(x)-u(y)\rvert \le \int_\gamma g\,ds\,,
\end{equation}
whenever both $u(x)$ and $u(y)$ are finite, and $\int_\gamma g\,ds =\infty$
otherwise. 
In addition, when $1\le p<\infty$, a measurable function $g\ge 0$ on $X$ is a {\em $p$-weak upper gradient}
of an extended real-valued function $u$ on $X$ 
if inequality~\eqref{e.modulus} holds for $p$-almost every curve $\gamma$ joining arbitrary points $x$ and $y$ in $X$;
that is, there exists a non-negative Borel function $\rho\in L^p(X)$ such that 
$\int_\gamma \rho\,ds=\infty$ whenever~\eqref{e.modulus} does not hold for the curve $\gamma$.
We refer to~\cite{MR2867756} for more information on $p$-weak upper gradients.

When $u$ is a (locally) Lipschitz function on $X$, 
the \emph{upper pointwise Lipschitz constant} of $u$ at $x\in X$
is defined as
\begin{equation}\label{e.L_constant}
\Lip(u,x)=\limsup_{r\to 0} \sup_{y\in B(x,r)} \frac{\lvert u(y)-u(x)\rvert}{r}\,.
\end{equation}
The Borel function $\Lip(u,\cdot)$
is an upper gradient of $u$; cf.~\cite[Proposition~1.14]{MR2867756}.  
Moreover, if $X$ is complete and $1<p<\infty$, then $\Lip(u,\cdot)$ is actually 
a so-called minimal $p$-weak upper gradient of $u$ 
(in particular, this implies that $\Lip(u,\cdot)\le g$ a.e.\ whenever $g\in L^p(X)$ is a
$p$-weak upper gradient of $u$).
This is a deep result of Cheeger, we refer to~\cite[Theorem~6.1]{MR1708448} and~\cite[p.~342]{MR2867756}.

\subsection{Poincar\'e inequalities}\label{s.poinc}
We say that the space $X$ supports a {\em  $(q,p)$-Poincar\'e inequality}, for $1\le q,p<\infty$,  
if there exist constants $C>0$ and $\lambda\ge 1$
such that for all balls $B\subset X$, all measurable
functions $u$ on $X$, and for all $p$-weak upper gradients $g$ of $u$,
\begin{equation}\label{e.poincare}
\biggl(\vint_B\lvert u-u_B\rvert^q\,d\mu\biggr)^{p/q} 
\le C\diam(B)^p \vint_{\lambda B} g^p\,d\mu\,. 
\end{equation} 
Here
\[
u_B=\vint_B u\,d\mu = \frac{1}{\mu(B)} \int_B u\,d\mu
\]
is the integral average of $u$ over the ball $B$, and
the left-hand side of~\eqref{e.poincare} is interpreted as $\infty$ whenever $u_B$ is not defined. 
We remark that $X$ supports a $(q,p)$-Poincar\'e inequality with constants $C>0$ and $\lambda\ge 1$ if, and only if,
inequality~\eqref{e.poincare} holds for all balls $B\subset X$, all functions $u\in L^1(X)$, and all upper
gradients $g$ of $u$; see \cite[Proposition~4.13]{MR1708448}.

If $1<p<\infty$ and $X$ supports a $(1,p)$-Poincar\'e inequality (and the measure $\mu$ is doubling,
as we assume throughout the paper),
then $X$ supports also a $(p,p)$-Poincar\'e inequality; see \cite[Corollary~4.24]{MR2867756}. 
If in addition $X$ is complete,
then there is an exponent $1< p_0<p$ such that $X$ supports
a $(p,p_0)$-Poincar\'e inequality 
and, consequently, also $(q,q)$-Poincar\'e inequalities with uniform constants whenever $p_0\le q\le p$;
for details we refer to~\cite{MR2415381} (see also \cite[Theorem~4.30]{MR2867756}) and to \cite[Theorem~4.21]{MR2867756}. Therefore the following (PI) condition, for a complete space $X$ 
supporting a $(1,p)$-Poincar\'e inequality,
is valid with the above exponents $1<p_0<p$.

However, since we do not in general assume that $X$ is complete, we
use in many of our results the following \emph{a priori} assumption
concerning the validity of (improved) Poincar\'e 
inequalities with uniform constants: 
\begin{itemize}
\item[(PI)] 
Let $1<p<\infty$ be given. We assume that there are $1<p_0<p$,
$C_P>0$ and $\tau \ge 1$ such that $X$ supports the $(q,q)$-Poincar\'e 
inequality
\begin{equation}\label{e.all_q}
\vint_B\lvert u-u_B\rvert^q\,d\mu\le C_{P}\diam(B)^q \vint_{\tau B} g^q\,d\mu
\end{equation}
for every $p_0\le q\le p$.
\end{itemize}
For simplicity, we will in the sequel use Poincar\'e inequalities 
with the constants  $C_P>0$ and $\tau \ge 1$.
Indeed, if only a $(1,p)$-Poincar\'e inequality is assumed, this is
just a matter of notation (in this case we may use both $(1,p)$-Poincar\'e and $(p,p)$-Poincar\'e inequality 
with the above constants). And if (PI) is assumed, the
$(1,q)$-Poincar\'e inequalities 
(with $C_P>0$  and $\tau \ge 1$) for $p_0\le q\le p$ are all trivial consequences of~\eqref{e.all_q}
and H\"older's inequality.

\subsection{Capacity and fatness}\label{s.cap}
Let $\Omega\subset X$ be a bounded open set and let $K\subset \Omega$ 
be a closed set.
We define the {\em (Lipschitz) variational $p$-capacity} of $K$ with respect to $\Omega$ to be
\begin{equation}\label{e.cap}
\cp_p(K,\Omega)=\inf\int_\Omega g^p\,d\mu\,,
\end{equation}
where the infimum is taken over all functions $u\in \Lip_0(\Omega)$,
such that $u\ge 1$ in $K$, and all $p$-weak upper gradients $g$ of $u$. 
If there are no such functions $u$, we set $\cp_p(K,\Omega)=\infty$.

\begin{remark}\label{e.equiv_cap}
If $\cp_p(K,\Omega)<\infty$, 
then the infimum in~\eqref{e.cap} can be restricted to
$u\in \Lip_0(\Omega)$ satisfying $\chi_K \le u\le 1$
and to $p$-weak upper gradients $g$ of $u$ such that $g=g\chi_{\Omega}\in L^p(X)$.
Indeed, if $u$ is an admissible test function for $\cp_p(K,\Omega)$ and $g$
is a $p$-weak upper gradient of $u$ such that $g\in L^p(\Omega)$, then 
$\tilde u=\max\{0,\min\{1,u\}\}$ belongs to $\Lip_0(\Omega)$ and
$\chi_K\le \tilde u\le 1$ on $X$. Moreover, the function $g$ is clearly a $p$-weak upper gradient of $\tilde u$.
By the glueing lemma \cite[Lemma~2.19]{MR2867756}, we may further assume that $g=0$ outside $\Omega$.
(Actually, since $g$ need not belong to $L^p(X)$ but this is needed in the glueing lemma, we first define a function
\[\tilde g=g\chi_{\Omega}+\Lip(\tilde u,\cdot)\chi_{X\setminus \Omega}\in L^p(X)\] that is a $p$-weak upper gradient
of $\tilde u$, cf.\ the proof of \cite[Theorem~2.6]{MR2867756}. Now the glueing lemma applies, with $\tilde g$,
showing that $g\chi_\Omega$ is a $p$-weak upper gradient of $\tilde u$.) 
\end{remark}

Let us  remark here that 
if the metric space $X$ is complete and supports a $(1,p)$-Poincar\'e inequality, then the above definition of $\cp_p(K,\Omega)$
is equivalent to the definition where the function $u$ is assumed to
belong to the \emph{Newtonian space} $N^{1,p}_0(\Omega)$.
However, we will not use the theory of Newtonian spaces in this paper, but rather 
refer to~\cite{MR2867756} for an introduction and basic properties
of Newtonian functions. In particular, 
see \cite[Theorem~6.19(x)]{MR2867756} for the above-mentioned equivalence of 
capacities in the complete case.

On the other hand, if $X=\R^n$, equipped with the Euclidean metric and the Lebesgue measure
(or more generally a $p$-admissible weight, see~\cite{MR1207810,MR1386213}),
then by standard approximation
\begin{equation}\label{e.cap_rn}
\cp_p(K,\Omega)=\inf\biggr\{\int_\Omega |\nabla u|^p\,dx : u\in C_0^\infty(\Omega),\ u\ge 1 \text{ in } K\biggr\}\,
\end{equation}
for all closed (compact) $K\subset\Omega$, and therefore $\cp_p(K,\Omega)$ is the usual
variational $p$-capacity of $K$.
In this case all our results (and computations) concerning
Lipschitz functions and their $p$-weak upper gradients (or upper pointwise Lipschitz
constants) can be restated using functions in $C_0^\infty(\Omega)$ and
the norms of their gradients. We recall that our approach is new even in this
special case.

We say that a closed set $E\subset X$ is {\em uniformly $p$-fat}, 
for $1\le p<\infty$, if there
exists a constant $0<c_0\le1$ such that
\begin{equation}\label{e.fatness}
\cp_p(E\cap \overline{B(x,r)},B(x,2r))\ge c_0 \cp_p(\overline{B(x,r)}, B(x,2r))
\end{equation}
for all $x\in E$ and all $0<r<(1/8)\diam(X)$.
If there exists a constant $r_0>0$ such that
condition~\eqref{e.fatness} holds for all
$x\in E$ and all $0<r<r_0$, the closed set $E$ is said to be
{\em locally uniformly $p$-fat}.

\begin{remark}
Both  Theorem \ref{t.self-impr} and Theorem \ref{t.self-impr-non-compl_Intro} are formulated in terms of
uniform fatness. However, the corresponding results
are valid also when `uniform fatness' is replaced
 by `local uniform fatness' (in the assumptions with exponent $p$ and in the conclusions
with exponent $q$).
In the sequel, we will exclusively focus on the case of uniformly fat sets.
The minor modifications (required
throughout the paper) in the local case are straightforward.
\end{remark}

The self-improvement of uniform $p$-fatness (that is formulated, e.g., in Theorem \ref{t.self-impr})
is  critical in various applications; examples 
beyond the scope of Hardy inequalities include global higher integrability of 
both the gradients of solutions to PDE's \cite{MR1302151,MR1386213} and
the upper gradients of certain quasiminimizers in metric measure spaces \cite{MR2418304}. In \cite{MR2564934} a quite simple proof
for the self-improvement of uniform $Q$-fatness is provided 
in the setting of Ahlfors $Q$-regular metric measure spaces.

In the Euclidean space $\R^n$, the self-improvement property is known to hold
also for more general $(\alpha,p)$-fatness conditions related to Riesz capacities
by the results of Lewis~\cite{MR946438}. For $\alpha=1$ these conditions are equivalent 
to the uniform $p$-fatness; cf.~\cite[p.~902]{MR1302151}.

In the rest of this paper (and hence in particular in our proof of the self-improvement of uniform fatness),
we only need the following two basic facts concerning the variational $p$-capacity, 
which hold under the assumption that the space $X$ supports a $(1,p)$-Poincar\'e inequality
(and hence also a $(p,p)$-Poincar\'e inequality).
First, there is a constant $C>0$ such that,
for each Lipschitz function $u$ on $X$, all $p$-weak upper gradients $g$ of $u$, and
for all balls $B\subset X$, we have
\begin{equation}\label{e.capacity_P}
\vint_B \lvert u\rvert^p\,d\mu \le \frac{C}{\cp_p(\overline{2^{-1}B}\cap \{u=0\}, B\}} \int_{\tau B} g^p\,d\mu\,.
\end{equation}
Here $\{u=0\}=\{x\in X\,:\, u(x)=0\}$ and $\tau$ is the dilatation 
from the $(p,p)$-Poincar\'e inequality~\eqref{e.all_q}.
This `capacitary Poincar\'e inequality' is
in the classical Euclidean case due to Maz'ya~\cite[Ch.~10]{MR817985}. 
For the metric
space version, cf.~\cite[Proposition~6.21]{MR2867756}.

The second fact is a comparison
between $p$-capacity and measure.
Namely, there is a constant $C>0$ such that for all
balls $B=B(x,r)$ with $0<r<(1/8)\diam(X)$ and for each 
closed set $E\subset \overline{B}$, 
\begin{equation}\label{e.comparison}
\frac{\mu(E)}{C\, r^p}\le \cp_p(E,2B)\le \frac{C_D\,\mu(B)}{r^p}\,;
\end{equation} 
see, for instance \cite[Proposition~6.16]{MR2867756}. The 
$(1,p)$-Poincar\'e inequality
is needed to ensure the validity of 
the lower bound in
inequality~\eqref{e.comparison}.

\subsection{Tracking constants}\label{s.constants}

Our results are based on 
quantitative estimates and absorption arguments, where
it is often crucial to track the dependencies of constants quantitatively.
For this purpose, we will use the following notational convention:
\begin{itemize}
\item
$C_{X,\ast,\dotsb,\ast}$ denotes a positive constant which quantitatively 
depends on the quantities indicated by the $\ast$'s and (possibly) on:
the doubling constant $C_D$ of the measure $\mu$ in~\eqref{e.doubling},
the constants $C_{P}$ and $\tau$ appearing in the $(q,q)$-Poincar\'e 
inequalities~\eqref{e.all_q} and
the constants appearing in
the capacitary Poincar\'e inequality~\eqref{e.capacity_P} and
the comparison inequality~\eqref{e.comparison}.
\end{itemize}
Observe that $C_{X,\ast,\dotsb,\ast}$ can implicitly depend on $p$
via the estimates in inequalities~\eqref{e.all_q}, \eqref{e.capacity_P} 
and~\eqref{e.comparison}.
However, any further dependencies on the exponent $p$ will be explicitly indicated.

\section{Improved boundary Poincar\'e inequalities}\label{s.self-impr}

{\emph{Recall, for the rest of the paper, that we assume 
$X$ to be a metric space (not necessarily complete)
equipped with a doubling measure $\mu$.
Further assumptions, concerning e.g.\ the validity of Poincar\'e
inequalities, will be stated separately in each of the following results.}

\medskip

Our proof of the self-improvement of uniform fatness is based on the following 
improved boundary Poincar\'e inequalities.

\begin{theorem}\label{t.main}
Let $1<p<\infty$
and suppose that $X$ supports  the improved $(q,q)$-Poincar\'e inequalities
(PI) for $p_0\le q\le p$. 
Assume that $E\subset X$ is a uniformly $p$-fat closed set.
Then there exists constants
$0<\varepsilon<p-p_0$ and  $C>0$, quantitatively, such that inequality
\begin{equation*}
\int_{B(w,\rho)} \lvert u\rvert^{p-\varepsilon}\,d\mu
\le C\rho ^{p-\varepsilon}\int_{B(w,\tau \rho)} \Lip(u,\cdot)^{p-\varepsilon}\,d\mu
\end{equation*}
holds whenever $w\in E$, $\rho>0$, and $u\in \Lip_0(X\setminus E)$.
\end{theorem}

\begin{proof}
Fix $w\in E$, a radius $\rho>0$, and
a function $u\in \Lip_0(X\setminus E)$. 
Clearly, we may assume that $\rho< (3/2)\diam(X)$.
It is convenient to write $r=\rho/(12\tau^2)$ and $B=B(w,r)$.
Let us assume, for the time being, that
$0<\varepsilon<p-p_0$ is given and
 $E_B\subset E\cap \overline{B}$ is {\em any} closed set such that $w\in E_B$.
Since
\begin{align*}
\int_{B(w,\rho)} \frac{\lvert u(x)\rvert^{p-\varepsilon}}{\rho^{p-\varepsilon}}\,d\mu(x)\le
\int_{B(w,\rho)\setminus E_{B}} \frac{\lvert u(x)\rvert^{p-\varepsilon}}{d_{E_{B}}(x)^{p-\varepsilon}}\,d\mu(x)\,,
\end{align*}
it suffices to find quantitative constants $0<\varepsilon<p-p_0$ and $C>0$
(and a closed set $E_{B}\subset E$ as above) such that 
\begin{equation}\label{e.impr}
\int_{B(w,\rho)\setminus E_{B}} \frac{\lvert u(x)\rvert^{p-\varepsilon}}{d_{E_{B}}(x)^{p-\varepsilon}}\,d\mu(x)
\le C\int_{B(w,\tau\rho)} \Lip(u,x)^{p-\varepsilon}\,d\mu(x)\,.
\end{equation}
We establish this improved local Hardy inequality
below in Proposition~\ref{p.improved}, and this proves the theorem.
Let us remark here that the proof of Proposition~\ref{p.improved} 
is rather involved and divided in Section~\ref{s.proof} to the following three stages:
`Truncation' in~\S\ref{s.truncation}, `Local Hardy' in~\S\ref{s.wannebo}, and `Improvement' in~\S\ref{s.improvement}.
\end{proof}

From Theorem~\ref{t.main} (that is based on postponed Proposition~\ref{p.improved}) we obtain the following estimate for
the capacity test-functions related to $\cp_p(E\cap \ol B,2B)$.
This estimate will be used in various settings to prove the self-improvement
of uniform fatness. 

\begin{proposition}\label{p.almostfat}
Let $1<p<\infty$
and suppose that $X$ supports the improved $(q,q)$-Poincar\'e inequalities
(PI) for $p_0\le q\le p$. 
Assume that $E\subset X$ is a uniformly $p$-fat closed set.
Then there exist constants $C>0$ and
$0<\varepsilon<p-p_0$, quantitatively, such that 
for all balls $B=B(w,R)$, with $w\in E$ and $0<R<(1/8)\diam(X)$, 
and for all functions $v\in \Lip_0(2B)$, with $0\le v\le 1$ and
$v=1$ in $E\cap \overline{B}$, it holds that
\begin{equation}\label{eq: goal}
\mu(B)R^{-(p-\eps)} \le C \int_{2B} \Lip(v,\cdot)^{p-\varepsilon}\,d\mu\,.
\end{equation}
\end{proposition}

\begin{proof}
This proof is based on a similar idea
as the proof of Lemma~2 in~\cite{MR2854110}.
Let $0<\varepsilon<p-p_0$ be given by Theorem~\ref{t.main}, and write $q=p-\varepsilon$
and $\ell=(2\tau)^{-1}\le 1/2$, where $\tau$ is the dilatation constant from the
$(q,q)$-Poincar\'e inequality~\eqref{e.all_q}.
Fix $w$, $R$, and $v$ as in the statement of the proposition.
The doubling inequality~\eqref{e.doubling} implies that there  is a constant 
$C_1=C_{C_D,\tau}>0$ such that
$\mu(\ell B)\ge C_1 \mu(B)$.
If $v_B >C_1/4$, we obtain from condition (PI) and the Sobolev inequality~\cite[Theorem~5.51]{MR2867756} 
for $v\in\Lip_0(2B)$ that
\[
C_1/4\le \vint_{B} |v|\,d\mu \le C_D \vint_{2B} |v|\,d\mu
\leq C R \bigg( \vint_{2B} \Lip(v,\cdot)^q\,d\mu\bigg)^{1/q},
\]
and from this~\eqref{eq: goal} follows easily.

We may hence assume that $v_B\le C_1/4$.
Let $\psi\in \Lip_0(B)$ be a cut-off function, defined as 
\[
\psi(x) 
= \max\Bigl\{0,1-\tfrac 2 R \dist\bigl(x,\tfrac 1 2 B\bigr)\Bigr\}\,,
\]
and take 
\[u=\min\{\psi,1-v\}.\]
Since $1-v=0$ in $E\cap \overline{B}$ and
$\psi=0$ in $X\setminus B$,   
we have that $u\in \Lip_0(X\setminus E)$.
Observe that  $u$ coincides with $1-v$ on $(1/2)B$, and therefore
$\Lip(u,\cdot)|_{(1/2)B}=\Lip(v,\cdot)|_{(1/2)B}$.

Let $F=\{  x\in \ell B : u(x) > 1/2 \}$.
We claim that $\mu(F) \ge  (C_1/2) \mu(B)$.
To prove this claim we assume the contrary, namely, that $\mu(F) < (C_1/2)\mu(B)$.
Since $v\ge 0$ and $v=1-u\ge 1/2$ in $\ell B\setminus F$, we obtain
from the assumptions $\mu(\ell B)\ge C_1 \mu(B)$ and  $\mu(F) < (C_1/2)\mu(B)$ that
\begin{align*}
\int_B v\,d\mu\ge \int_{\ell B\setminus F} v\,d\mu
&\ge \tfrac 1 2\bigl(\mu(\ell B)-\mu(F)\bigr)\\
&> \tfrac 1 2 \bigl(C_1\mu(B)-(C_1/2)\mu(B)\bigr) = \tfrac 1 4 C_1\mu(B)\,.
\end{align*}
This contradicts the assumption $v_B\le C_1/4$, and thus indeed $\mu(F) \ge (C_1/2) \mu(B)$.

Theorem~\ref{t.main}, with $\rho=\ell R$, now implies that
\begin{align*}
(C_1/2) \mu(B) \le \mu(F)
\le  2^{q}\int_{\ell B}|u|^q\,d\mu
\le C R^q \int_{\tau\ell B} \Lip(u,\cdot)^q\,d\mu \le C R^q \int_{2 B} \Lip(v,\cdot)^q\,d\mu\,.
\end{align*}
This proves estimate~\eqref{eq: goal} and concludes the proof.
\end{proof}

In a Euclidean space $\R^n$, which supports the  $(1,p)$-Poincar\'e inequalities for all 
$1\le p<\infty$,
Proposition~\ref{p.almostfat} yields immediately the self-improvement
of uniform $p$-fatness.  Indeed, we can replace in our argument the 
Lipschitz function $v\in\Lip_0(2B)$ with a function $\tilde v\in C_0^\infty(2B)$ 
and the pointwise Lipschitz constant $\Lip(v,\cdot)$ with $|\nabla \tilde v|$,
whence the uniform $(p-\eps)$-fatness of $E$ follows 
from estimates~\eqref{e.comparison} and ~\eqref{eq: goal}.

More generally, in a complete metric space $X$ supporting a
$(1,p)$-Poincar\'e inequality, we can deduce the self-improvement
of uniform fatness from Proposition~\ref{p.almostfat} with the help of 
some deep facts concerning analysis on
metric spaces (see the proof below). Nevertheless, with an additional
argument using the interplay between uniform fatness and 
density conditions for suitable Hausdorff contents, it is possible to
obtain a version of the self-improvement in a non-complete setting as well
(Theorem~\ref{t.self-impr-non-compl_Intro}),
and hence in particular without the use of Cheeger's differentiation theory, but then the
$(p,p_0)$-Poincar\'e inequality, or at least the validity of improved Poincar\'e inequalities (PI)
for $p_0\le q\le p$, has to be explicitly assumed for some exponent $1<p_0<p$;
see Section~\ref{s.Hausd} for details.

\begin{proof}[Proof of Theorem~\ref{t.self-impr}]
Since the space $X$ is assumed to be complete, the validity of the improved Poincar\'e
inequalities (PI) follows from the $(1,p)$-Poincar\'e inequality, as
discussed in Section~\ref{s.poinc}. Hence we can apply 
Proposition~\ref{p.almostfat}. Moreover, by the deep result of Cheeger,
\cite[Theorem~6.1]{MR1708448} (see also \cite[Theorem A.7]{MR2867756}), the upper pointwise
Lipschitz constant $\Lip(v,\cdot)$ is a minimal
$(p-\eps)$-weak gradient of the Lipschitz function $v$, and so
we obtain from estimates~\eqref{e.comparison} and ~\eqref{eq: goal} (and Remark \ref{e.equiv_cap})
that the set $E$ is indeed uniformly $(p-\eps)$-fat.
\end{proof}

\section{Self-improvement of uniform fatness in non-complete spaces}\label{s.Hausd}

In this section we provide the additional argument that is needed
for the proof of the self-improvement result in the setting of non-complete metric spaces,
Theorem~\ref{t.self-impr-non-compl_Intro}.
In fact, we prove the following slightly stronger result
(by H\"older's inequality, the $(p,p_0)$-Poincar\'e inequality 
that was assumed in Theorem~\ref{t.self-impr-non-compl_Intro} implies
the improved Poincar\'e inequalities (PI) for $p_0\le q\le p$.) 
 
\begin{theorem}\label{t.self-impr-non-compl}
Let $1<p<\infty$ and suppose that $X$ supports the improved $(q,q)$-Poincar\'e inequalities
(PI) for $p_0\le q\le p$. 
Assume that $E\subset X$ is a  uniformly $p$-fat closed set
and that $E\cap \ol{B(w,r)}$ is compact for all $w\in E$ and all $r>0$.
Then there exists 
$0<\varepsilon<p-p_0$ such that 
$E$ is uniformly $(p-\varepsilon)$-fat; here both $\varepsilon$ and
the constant of uniform $(p-\varepsilon)$-fatness are quantitative.
\end{theorem} 

The proof of Theorem \ref{t.self-impr-non-compl} is based on
Proposition \ref{p.almostfat}, but
we also need some auxiliary results related to
Hausdorff contents.
Note that these auxiliary results are essentially established in~\cite{MR2854110},
but there the space $X$ is assumed to be complete.

The {\it Hausdorff content of codimension $q$}
of a set $K\subset X$ is defined by  
\[
\wtilde\Ha^q_\rho(K)=\inf\biggl\{\sum_{k} \mu(B(x_k,r_k))\,r_k^{-q} :
K\subset\bigcup_{k} B(x_k,r_k),\ x_k\in K,\ 0<r_k\leq \rho \biggr\}\,.
\]
Density conditions for these Hausdorff contents are known to be closely related
to uniform fatness. Indeed, from Proposition~\ref{p.almostfat} we obtain the following
result.

\begin{lemma}\label{l.thick}
Let $1<p<\infty$ 
and suppose that $X$ supports the improved $(q,q)$-Poincar\'e inequalities
(PI) for $p_0\le q\le p$. 
Assume that $E\subset X$ is a uniformly $p$-fat closed set.
Then there exist constants $C>0$ and $p_0<q<p$, quantitatively, such that
\begin{equation}\label{e.thick}
\wtilde\Ha^q_{R/2} \bigl(E\cap \ol{B(w,R)}\bigr) \ge C\mu\bigl(B(w,R)\bigr)R^{-q}
\end{equation}
whenever $w\in E$ and $0<R<(1/8)\diam(X)$ are such
that $E\cap \ol{B(w,R)}$ is compact.
\end{lemma}

\begin{proof}
Fix $w\in E$ and $0<R<(1/8)\diam(X)$, write $B=B(w,R)$, and assume that $E\cap \ol B$ is compact.
Let $\{B_k\}$, where $B_k=B(x_k,r_k)$ with $x_k\in E\cap \ol B$ and $0<r_k\le R/2$, 
be a cover of $E\cap \ol B$. Since $E\cap \ol B$ is compact, we may assume that this
cover is finite, i.e.\ $E\cap \ol B\subset \bigcup_{k=1}^N B_k$.
Also let $q=p-\eps$, where $\eps$ is as in Proposition~\ref{p.almostfat}.

Define
\[
v(x)=\max_{1\leq k\leq N}\big\{0,1- {r_k}^{-1}\dist(x,B_k)\big\}.
\]
Then $v$ is a Lipschitz function, 
$v = 1$ in $E\cap \ol B$, $v=0$ outside $2B$, and $0\le v\le 1$.
Moreover, the upper pointwise Lipschitz constant of $v$ satisfies
$\Lip(v,x)\le \max_{1\leq k\leq N} {r_k}^{-1}\chi_{\ol{2B_k}}(x)$ 
for all $x\in X$, 
and hence 
\begin{equation*}\label{eq: grad}
\Lip(v,x)^{q} \le \sum_{k=1}^N {r_k}^{-q}\chi_{\ol{2B_k}}(x)
\end{equation*}
for all $x\in 2B$.
Thus we obtain from Proposition~\ref{p.almostfat} (and the doubling condition) that
\begin{equation*}
\mu(B)R^{-q} \le C \int_{2B} \Lip(v,x)^{q}\,d\mu(x) \le 
C \sum_{k=1}^N \mu\bigl(\ol{2B_k}\bigr) {r_k}^{-q} \le 
C \sum_{k=1}^N \mu(B_k) {r_k}^{-q}\,.
\end{equation*}
Taking the infimum over all such covers of $E\cap \ol B$ yields the claim.
\end{proof}

On the other hand, from~\eqref{e.thick} we get back to $t$-uniform fatness, for any $t>q$.

\begin{lemma}\label{l.fat}
Let $1<q<\infty$
and suppose that $X$ supports a $(1,t)$-Poincar\'e inequality for all $t>q$.
Let $E\subset X$ be a closed set.
If there exists $C>0$ such that the density condition
\begin{equation}\label{e.thick_again}
\wtilde\Ha^q_{R/2} \bigl(E\cap \ol{B(w,R)}\bigr) \ge C\mu\bigl(B(w,R)\bigr)R^{-q}
\end{equation}
holds for all $w\in E$ and all $0<R<(1/8)\diam(X)$, then $E$ is uniformly $t$-fat for all $t>q$. 
\end{lemma}

\begin{proof}
Fix $t>q$. Let $w\in E$ and $0<R< (1/8)\diam(X)$, write $B=B(w,R)$, and let $u\in\Lip_0(2B)$ 
be such that $0\le u\le 1$ and $u=1$ in $E\cap \overline{B}$.
By the capacity comparison estimate~\eqref{e.comparison} and Remark~\ref{e.equiv_cap}, 
it suffices to show that there exists a constant $C>0$, independent of
$w$, $R$ and $u$, such that
\begin{equation}\label{eq: goal again}
\mu(B)R^{-t} \le C \int_{2B} g^{t}\,d\mu
\end{equation}
for all $t$-weak upper gradients $g$ of $u$ such that $g=g\chi_{2B}\in L^t(X)$.

If $u_{2B}\ge 1/2$, then it follows from the 
Sobolev inequality~\cite[Theorem~5.51]{MR2867756} 
that
\[
1/2 \le \vint_{2B} u\,d\mu
\leq C_{X,p}R \bigg(\,\vint_{2B} g^t\,d\mu\bigg)^{1/t},
\]
and from this~\eqref{eq: goal again} follows easily.

On the other hand, if $u_{2B} < 1/2$, we can use similar reasoning as 
in~\cite[p.~729]{MR2854110} 
(which is based on the proof of~\cite[Theorem~5.9]{MR1654771}),
but let us recall the main steps for convenience.
Since $t>q$ and $1/2<u(x)-u_{2B}=|u(x)-u_{2B}|$ for each $x\in E\cap \overline{B}$,
we can apply a well-known chaining argument 
(using also the continuity of $u$ and the $(1,t)$-Poincar\'e inequality)
to find for
each $x\in E\cap \overline{B}$ a ball $B_x=B(x,r_x)$ with $0<r_x\le 3R$
such that 
\begin{equation}\label{e.good ball}
 \mu(B_x)r_x^{-q}\le C_{X,t,q} R^{t-q}\int_{\tau B_x} g^t\, d\mu\,.
\end{equation}

The $5r$-covering lemma then yields us a countable collection of
points $x_1,x_2,\ldots\in E\cap \ol B$ such that the 
corresponding balls $B_k=\tau B_{x_k}$ are pairwise disjoint, but
the balls $5 B_k$ cover $E\cap \ol{B}$. Using the assumption~\eqref{e.thick_again} for this particular cover 
and the doubling property of $\mu$, we find that
\begin{equation}
\mu(B)R^{-q}\le C \sum_k \mu(B_{x_k})r_{x_k}^{-q}\,,
\end{equation}
whence estimate~\eqref{e.good ball} and the pairwise disjointness of the balls
$B_k$ yield the claim~\eqref{eq: goal again}.  
(In particular, here we may assume that the radii of the balls $5 B_{k}$ are all
less than $R/2$, since otherwise the claim readily follows from
the doubling property of $\mu$ and inequality~\eqref{e.good ball} 
applied to a ball $B_{x_k}$ with $5\tau r_{x_k} > R/2$).
\end{proof}
 
\begin{proof}[The proof of Theorem~\ref{t.self-impr-non-compl}]
Since we assumed that $E$ is uniformly $p$-fat and $E\cap \ol{B(w,R)}$ is compact for all
$w\in E$ and all $R>0$, we have by Lemma~\ref{l.thick} that
\[\wtilde\Ha^q_{R/2} \bigl(E\cap \ol{B(w,R)}\bigr) \ge C\mu\bigl(B(w,R)\bigr)R^{-q}\]
for all $w\in E$ and all $0<R<(1/8)\diam(X)$, where $p_0<q<p$. 
But now we can fix $q<t < p$, and
Lemma~\ref{l.fat} yields that $E$ is uniformly $t$-fat. Notice, in particular, that the
$(1,t)$-Poincar\'e inequality that is needed in the proof of Lemma~\ref{l.fat} is valid
by the assumed improved Poincar\'e inequalities (PI) since $p_0<t<p$. 
\end{proof}

\section{Improved local Hardy inequalities}\label{s.proof}

This section is devoted to the proof of the improved
local Hardy inequality~\eqref{e.impr}
that  is reformulated as  Proposition~\ref{p.improved}.
The proof of this proposition is divided in the following three parts.
In \S\ref{s.truncation} we prepare for the localization 
of Hardy inequalities by
truncating the set $E$ and
proving a local absorption lemma.
In \S\ref{s.wannebo} we then obtain
localized Hardy inequalities with exponent $p$, and in \S\ref{s.improvement} we
finally establish their self-improvement.

\subsection{Truncation}\label{s.truncation}

We begin with some technical tools that will be needed in the
proofs of the local Hardy inequalities.
The following truncation procedure provides us with the closed set $E_B\subset \overline{B}$
that was required in the proof of Theorem~\ref{t.main}.
A similar procedure was used in~\cite[p. 180]{MR946438}
when proving the self-improvement of uniform $(\alpha,p)$-fatness conditions
in $\R^n$, and later also in~\cite{MR1386213}, for weighted $\R^n$, 
and in~\cite{MR1869615}, for general metric spaces. 

We write $\N=\{1,2,3,\ldots\}$ and $\N_0=\N\cup\{0\}$.

\begin{lemma}\label{l.truncation}
 Assume that $E\subset X$ is a closed set and that $B=B(w,r)$ for $w\in E$ and $r>0$.
Let $E_{B}^0=E\cap \overline{\frac 1 2 B}$, define inductively, for every $j\in\N$, that
\[
E_{B}^{j}=\bigcup_{x\in E_{B}^{j-1}} E\cap \overline{B(x,2^{-j-1}r)}\,,\quad
\text{ and set } \quad
E_{B}=\overline{\bigcup_{j\in \N_0} E_{B}^{j}}.
\]
Then the following statements hold:
\begin{itemize}
\item [(a)] $w\in E_{B}$
\item[(b)] $E_{B}\subset E$
\item[(c)] $E_{B}\subset \overline{B}$
\item[(d)] $E^{j-1}_{B}\subset E^{j}_{B}\subset E_{B}$  for every $j\in\N$.
\end{itemize}
\end{lemma}

\begin{proof}
Part (a) is is true since $w\in E_{B}^0$. Part (b) follows from the facts
that $E$ is closed and $\cup_j E_{B}^j\subset E$ by definition.
To verify (c), we fix $x\in E_{B}^j$.
If $j=0$, then $x\in \overline{B}$. If $j>0$, then
by induction we find a sequence $x_j, \ldots,x_0$ such $x_j=x$ and,
for
each $k=0,\ldots,j$,
$x_k\in E^k_{B}$ and
$x_k\in E\cap \overline{B(x_{k-1},2^{-k-1}r)}$ if $k>0$. It follows that
\[
d(x,w)\le \sum_{k=1}^j  d(x_k,x_{k-1}) + d(x_0,w) \le \sum_{k=1}^j 2^{-k-1}r + 2^{-1}r < r\,.
\]
Hence, $x\in B(w,r)\subset \overline{B}$.
We have shown that $E_{B}^j\subset \overline{B}$ whenever $j\in\N_0$,
whence it follows that also $E_{B}\subset \overline{B}$.
To prove (d) we fix $j\in\N$ and $x\in E^{j-1}_{B}$. By definition
we have $x\in E$ and, hence,
$x\in E\cap B(x,2^{-j-1}r)\subset E^j_{B}$.
\end{proof}

Next we show that Lemma \ref{l.truncation}, in fact, truncates the set $E$ to $B$ in such a way 
that there are always certain balls  $\widehat{B}$
whose intersection with the truncated set $E_B$
contain big pieces of the original set $E$ (by these balls
we later employ the uniform fatness of $E$).

\begin{lemma}\label{l.pallot}
Let $E$, $B$, and $E_{B}$ be as in Lemma~\ref{l.truncation}.
Suppose that  $m\in\N_0$ and $x\in X$ is such that
$d_{E_{B}}(x) <2^{-m+1}r$.
Then there exists
a ball $\widehat{B}=B(y_{x,m},2^{-m-1}r)$ such that $y_{x,m}\in E$,
\begin{equation}\label{e.bubble} 
\overline{2^{-1}\widehat{B}}\cap E=\overline{2^{-1}\widehat{B}}\cap E_{B}\,,
\end{equation}
and $\sigma \widehat{B}\subset B(x,\sigma 2^{-m+2}r)$ for every $\sigma\ge 1$.
\end{lemma} 

\begin{proof}
In this proof we will apply Lemma~\ref{l.truncation} several times without further notice.
Since
$d_{E_{B}}(x)<2^{-m+1}r$ there exists 
$y\in \cup_{j\in\N_0}E_{B}^j\subset E$ such that
$d(y,x)<2^{-m+1}r$. Let us fix $j\in\N_0$ such that
$y\in E_{B}^j$. There are two cases to be treated.

First, let us consider the case when $j>m\ge 0$. By induction, there
are points $y_k\in E_{B}^k$ with $k=m,\ldots,j$ such that 
$y_j=y$ and
$y_k\in E\cap \overline{B(y_{k-1},2^{-k-1}r)}$
for every $k=m+1,\ldots,j$. It follows that
\[
d(y_m,y) = d(y_j,y_m) \le \sum_{k=m+1}^j d(y_k,y_{k-1})
\le \sum_{k=m+1}^j 2^{-k-1}r < 2^{-m-1}r\,.
\]
Take $y_{x,m}=y_m\in E_{B}^m\subset E$ and $\widehat{B}=B(y_m,2^{-m-1}r)$. 
If $\sigma\ge 1$ and  $z\in \sigma \widehat{B}$, then
\begin{align*}
d(z,x) &\le d(z,y_m) + d(y_m,y) + d(y,x)\\
&\le \sigma 2^{-m-1}r + 2^{-m-1} r+ 2^{-m+1}r<\sigma 2^{-m+2}r\,,
\end{align*}
and thus $\sigma \widehat{B}\subset B(x, \sigma 2^{-m+2}r)$.
Moreover, since $y_m\in E_{B}^m$, we have
\begin{align*}
\overline{2^{-1}\widehat{B}}\cap E & =E\cap \overline{B(y_m, 2^{-m-2}r)}
\subset \bigcup_{z\in E_{B}^m} E\cap \overline{B(z, 2^{-m-2}r)} = E_{B}^{m+1}\subset E_{B}\,.
\end{align*}
On the other hand $E_{B}\subset E$, and thus 
$\overline{2^{-1}\widehat{B}}\cap E=\overline{2^{-1}\widehat{B}}\cap E_{B}$.

Let us then consider the case $m\ge j\ge 0$. We take $y_{x,m}=y\in E$ and $\widehat{B}=B(y,2^{-m-1}r)$.
Then, for every $\sigma\ge1$ and each 
$z\in \sigma \widehat{B}$,
\[
d(z,x)\le d(z,y) + d(y,x) < \sigma 2^{-m-1}r + 2^{-m+1}r < \sigma 2^{-m+2}r\,,
\]
and so $\sigma \widehat{B}\subset B(x, \sigma 2^{-m+2}r)$. 
Since $y\in E_{B}^j\subset E_{B}^m\subset E_{B}$ 
we can repeat the argument above, with
$y_m$ replaced by $y$,
and it follows as above that $\overline{2^{-1}\widehat{B}}\cap E=\overline{2^{-1}\widehat{B}}\cap E_{B}$.
\end{proof}

One of the reasons for truncating the set $E$,  {\em in the first place}, is to obtain
the absorption Lemma~\ref{l.absorb}.
This lemma is needed twice during the rest of the paper,
with slightly different contexts, and hence there are two
different assumptions concerning the validity of Poincar\'e inequalities.    
The dependencies 
of the constants below are rather delicate, and it
is important to track them carefully; to this end, recall
our notational convention from~\S\ref{s.constants}.

\begin{lemma}\label{l.absorb}
Suppose that either
\begin{itemize}
 \item[{\em(i)}] $1\le q=p<\infty$ and $X$ supports a $(1,p)$-Poincar\'e inequality; or
 \item[{\em(ii)}] $1<p_0 < p<\infty$ and $X$ supports the improved Poincar\'e inequalities 
                  (PI) for exponents $p_0 \le q \le p$.
\end{itemize}
In addition, let $E$, $B$, and $E_{B}$ be as in Lemma~\ref{l.truncation},
let $\sigma\ge 1$ and $\varsigma\ge 2$, and write $B^*=\varsigma B$.
Assume that $u\in\Lip(X)$ is such that $u = 0$ on $E_{B}$,
and that $g$ is a $q$-weak upper gradient of $u$ such that inequality
\[
\int_{B^*\setminus E_{B}} \frac{\lvert u(x)\rvert^{q}}{d_{E_{B}}(x)^{q}}\,d\mu(x)
\le C_1\int_{\sigma B^* \setminus E_{B}} \frac{\lvert u(x)\rvert^{q}}{d_{E_{B}}(x)^{q}}\,d\mu(x)
+ C_2\int_{\sigma B^*} g(x)^{q}\,d\mu(x)
\]
holds with some constants $C_1,C_2>0$.
Then
\begin{align*}
C_3\int_{\sigma B^*\setminus E_{B}} \frac{\lvert u(x)\rvert^{q}}{d_{E_{B}}(x)^{q}}\,d\mu(x)
\le C_4\int_{\tau \sigma B^*} g(x)^{q}\,d\mu(x)\,,
\end{align*}
where $C_3=1-C_1(1+C_{C_D,\sigma,\varsigma,p})$ and $C_4=(1+C_2)C_{X,\sigma,\varsigma,q}$.
\end{lemma}

\begin{proof}
Since $E_{B}\subset \overline{B}$ and $\varsigma\ge 2$, we obtain the estimate
\begin{align*}
\int_{(\sigma B^*\setminus E_{B})\setminus B^*} & \frac{\lvert u(x)\rvert^{q}}{d_{E_{B}}(x)^{q}}\,d\mu(x)
\le r^{-q}\int_{\sigma B^*}\lvert u(x)\rvert^q\,d\mu(x)\\
& \le 3^q r^{-q} \bigg( \int_{\sigma B^*}\lvert u(x)-u_{\sigma B^*}\rvert^q\,d\mu(x) + \mu(\sigma B^*)\lvert u_{\sigma B^*}-u_{B^*}\rvert^q + \mu(\sigma B^*) \lvert u_{B^*}\rvert^q \bigg)\,.
\end{align*}
By the doubling inequality~\eqref{e.doubling} and the $(q,q)$-Poincar\'e inequality~\eqref{e.all_q}
(recall that in case (i), i.e.\ for $q=p$, this is a consequence of the $(1,p)$-Poincar\'e inequality, 
cf.\ \cite[Corollary~4.24]{MR2867756})
we obtain
\begin{align*}
3^q r^{-q}\bigg(\int_{\sigma B^*} &\lvert u(x)-u_{\sigma B^*}\rvert^q\,d\mu(x) + 
    \mu(\sigma B^*)\lvert u_{\sigma B^*}-u_{B^*}\rvert^q\bigg) \\
&\le C_{\sigma,q,C_D}r^{-q}\int_{\sigma B^*}\lvert u(x)-u_{\sigma B^*}\rvert^q\,d\mu(x)
    \le C_{X,\sigma,\varsigma,q} \int_{\tau \sigma B^*} g(x)^q\,d\mu(x)\,.
\end{align*}
On the other hand, by the assumption,
\begin{align*}
3^qr^{-q}&\mu(\sigma B^*) \lvert u_{B^*}\rvert^q 
\le 3^q C_{\sigma,C_D}r^{-q}\int_{B^*\setminus E_{B}}\lvert u(x)\rvert^q\,d\mu(x)\\
&\le 3^q \varsigma^{q} C_{\sigma,C_D}\int_{B^*\setminus E_{B}}\frac{\lvert u(x)\rvert^q}{d_{E_{B}}(x)^q}\,d\mu(x)\\
&\le  3^p \varsigma^{p} C_{\sigma,C_D}C_1  \int_{\sigma B^*\setminus E_{B}} \frac{\lvert u(x)\rvert^{q}}{d_{E_{B}}(x)^{q}}\,d\mu(x)
+ 3^q \varsigma^{q} C_{\sigma,C_D}C_2 \int_{\sigma B^*} g(x)^{q}\,d\mu(x)\,.
\end{align*}
Combining the estimates above, we find that
\begin{align*}
\int_{\sigma B^*\setminus E_{B}} &\frac{\lvert u(x)\rvert^q}{d_{E_{B}}(x)^q}\,d\mu(x) 
= \int_{B^*\setminus E_{B}} \frac{\lvert u(x)\rvert^q}{d_{E_{B}}(x)^q}\,d\mu(x) +
  \int_{(\sigma B^*\setminus E_{B})\setminus B^*}\frac{\lvert u(x)\rvert^q}{d_{E_{B}}(x)^q}\,d\mu(x) \\
&\le C_1(1+C_{C_D,\sigma,\varsigma,p})\int_{\sigma B^*\setminus E_{B}} \frac{\lvert u(x)\rvert^{q}}{d_{E_{B}}(x)^{q}}\,d\mu(x) 
    + C_4\int_{\tau \sigma B^*} g(x)^{q}\,d\mu(x)\,,
\end{align*}
where $C_4=(1+C_2)C_{X,\sigma,\varsigma,q}$. This concludes the proof.
\end{proof}

\subsection{Local Hardy}\label{s.wannebo}

In this section we prove Proposition \ref{p.Wannebo} that gives a local $p$-Hardy inequality with respect to
the truncated set $E_{B}$ (Theorem~\ref{t.Wannebo_intro} is also proved at the end of this section).  
This is done by adapting the proof of~\cite[Theorem~3]{MR2854110}
to the present setting. The proof in~\cite{MR2854110} is,
in turn, based on the ideas of Wannebo~\cite{MR1010807}; see also~\cite{MR2114252}.

Throughout this section, we will assume that $X$ supports a $(1,p)$-Poincar\'e inequality,
whence $X$ supports also a $(p,p)$-Poincar\'e inequality. Both of these
inequalities are assumed to be valid with constants $C_P>0$ and $\tau\ge 1$.

\begin{proposition}\label{p.Wannebo}
Let $1<p<\infty$ and suppose 
that $X$ supports a $(1,p)$-Poincar\'e inequality.
Assume that $E\subset X$ is a uniformly $p$-fat closed set, let
$w\in E$ and $0<r< (1/8) \diam(X)$, and let $E_{B}$ be as in Lemma~\ref{l.truncation}
for $B=B(w,r)$.
Let $u\in \Lip(X)$ and let $g$ be a $p$-weak upper gradient of $u$
such that $u = 0 = g$ in an open set $U\subset X$ satisfying the condition $\dist(E_{B},X\setminus U)>0$
(or the condition $X=U$). Then  
\begin{equation}\label{e.int}
\int_{8\tau B\setminus E_{B}}\frac{\lvert u(x)\rvert^p}{d_{E_{B}}(x)^p}\,d\mu(x) \le C_H\int_{8\tau^2 B} g(x)^p\,d\mu(x)\,.
\end{equation}
Here $C_H=C_{X,p,c_0}$ and the number $0<c_0\le 1$ is the constant from
the uniform $p$-fatness condition~\eqref{e.fatness} for $E$.
\end{proposition}

Recall that we do not assume $X$ to be complete,
and hence it is not necessarily true
that $\dist(K,X\setminus U)>0$
whenever $K$ is a closed subset of a bounded open set $U$.
For this reason we make in Proposition~\ref{p.Wannebo} the explicit 
assumption that $\dist(E_{B},X\setminus U)>0$.

The proof of Proposition~\ref{p.Wannebo} is based
upon covering and absorption arguments, and it 
will be completed at the end of this section.
We begin with Lemmata~\ref{l.aux1} and~\ref{l.aux2}
that provide information concerning the individual balls in the following covering families.
For the rest of this section 
we assume that $p$, $X$, $E$, $B=B(w,r)$, and $E_{B}$ are as in Proposition~\ref{p.Wannebo}
(these are considered arbitrary but fixed).

For each $m\in\Z$, let us write  
\[
G_m= \{x\in 8\tau B \,:\, 2^{-m}r\le d_{E_{B}}(x) <2^{-m+1}r\}
\] 
and
\[
\widetilde{G}_m = \bigcup_{k=m}^\infty G_k=\{x\in 8\tau B\,:\, 0< d_{E_{B}}(x) <2^{-m+1}r\}\,.
\]  
For every $m\in\N_0$ we let $\mathcal{G}_m$ be a (countable) cover of $G_m$ with open balls $\widetilde{B}$ 
that are centered at $G_m$ and of radius $2^{-m+2}r$. 
Moreover, we require that 
$\{2^{-1}\widetilde{B}\,:\,\widetilde{B}\in\mathcal{G}_m\}$ is a disjoint family,
whence there exits $C=C_{C_D,\tau}>0$ such that  
\begin{equation}\label{e.finover}
 \sum_{\widetilde{B}\in \mathcal{G}_m} \chi_{\tau \widetilde{B}}\le C\,.
 \end{equation}  
The existence of such a cover follows using a maximal packing argument and
the doubling property of $\mu$. 

\begin{lemma}\label{l.aux1}
Let us define $\ell = \lceil \log_2(\tau)\rceil+2$. Then, for each 
$m\in \N_0$ and every ball 
$\widetilde{B}\in\mathcal{G}_m$, we have
\begin{equation}\label{e.ees}
\tau \widetilde B\setminus E_{B}\subset \widetilde{G}_{m-\ell}\,.
\end{equation}
\end{lemma}

\begin{proof}
Fix $m\in\N_0$ and  
$\widetilde{B}\in\mathcal{G}_m$.
By definition, we have $\widetilde{B}=B(x_{\widetilde{B}}, 2^{-m+2}r)$ with $x_{\widetilde{B}}\in G_m$.
Let $x\in \tau \widetilde{B}\setminus E_{B}$. Then $d_{E_{B}}(x)>0$. 
Moreover, 
\begin{align*}
d_{E_{B}}(x)=\dist(x,E_{B})&\le d(x,x_{\widetilde{B}}) + \dist(x_{\widetilde{B}},E_{B})\\
& < \tau 2^{-m+2}r + 2^{-m+1} r
< \tau 2^{-m+3}r \le  2^{-(m-\ell)+1}r\,.
\end{align*}
Since $m\ge 0$, a modification of the previous estimate also yields
\begin{align*}
d(x,w) &\le \dist(x,\overline{B})+ r \le \dist(x,E_{B}) + r < 4\tau r + 2r + r <8\tau r\,,
\end{align*}
and it follows that $x\in B(w,8 \tau r)=8\tau B$. We can now conclude 
that $x\in \widetilde{G}_{m-\ell}$.
\end{proof}

The uniform $p$-fatness of $E$ is exclusively used in the following lemma.

\begin{lemma}\label{l.aux2}
Let $v$ be a Lipschitz function on $X$ such that 
$v = 0$ on $E_{B}$ and let $g$ be a $p$-weak upper gradient of $v$.
Then, for every $m\in\N_0$ and each $\widetilde{B}\in\mathcal{G}_m$,
\begin{equation}\label{e.single}
\int_{\widetilde{B}}\lvert v(x)\rvert^p\,d\mu(x)
\le \frac{C_{X,p}}{c_0} 2^{-mp} r^p\int_{\tau \widetilde{B}}g(x)^p\,d\mu(x)\,.
\end{equation} 
\end{lemma}

\begin{proof}
Fix $m\in\N_0$ and $\widetilde{B}=B(x_{\widetilde{B}}, 2^{-m+2}r)\in \mathcal{G}_m$. 
Then, by definition, $x_{\widetilde{B}}\in G_m$.
We apply Lemma \ref{l.pallot} and thereby associate to $\widetilde{B}$ a
smaller open ball $\widehat{B}\subset\widetilde{B}$, centered at $E$ and of radius
$2^{-m-1}r<(1/8)\diam(X)$.
Note first that
\begin{align*}
\vint_{\widetilde{B}}\lvert v(x)\rvert^p\,d\mu(x) 
&\le C_p\bigg(\vint_{\widetilde{B}}\lvert v(x)-v_{\widetilde{B}}\rvert^p\,d\mu(x) + 
  \lvert v_{\widetilde{B}} -v_{\widehat{B}}\lvert^p + \lvert v_{\widehat{B}}\rvert^p \bigg)\,.
\end{align*}
Here, by  H\"older's inequality and the doubling condition \eqref{e.doubling},
\begin{align*}
 \lvert v_{\widetilde{B}} -v_{\widehat{B}}\lvert^p&\le 
\bigg( \intav_{\widehat{B}} \lvert v(x)-v_{\widetilde{B}}\rvert\,d\mu(x) \bigg)^p
\le C_{C_D}\vint_{\widetilde{B}}\lvert v(x)-v_{\widetilde{B}}\rvert^p\,d\mu(x)\,,
\end{align*}
and therefore, by the $(p,p)$-Poincar\'e inequality, 
we have that
\begin{align*}
C_p\bigg(\vint_{\widetilde{B}}\lvert v(x)-v_{\widetilde{B}}\rvert^p\,d\mu(x) + 
     \lvert v_{\widetilde{B}} -v_{\widehat{B}}\lvert^p \bigg)
\le C_{X,p} 2^{-mp}r^p\vint_{\tau \widetilde{B}}g(x)^p\,d\mu(x)\,.
\end{align*} 
On the other hand, by the capacitary Poincar\'e inequality~\eqref{e.capacity_P}, 
\begin{align*}
\lvert v_{\widehat{B}}\rvert^p\le\intav_{\widehat{B}}\lvert v(x)\rvert^p\,d\mu(x)
\le \frac{C_{X}}{\cp_p(\overline{2^{-1}\widehat{B}}\cap \{v=0\},\widehat B)} \int_{\tau\widehat{B}}g(x)^p\,d\mu(x)\,.
\end{align*}
Recall that
$v(x)=0$ whenever $x\in E_{B}$ (by assumption) and
 $\overline{2^{-1}\widehat{B}}\cap E_{B}=\overline{2^{-1}\widehat{B}}\cap E$ by Lemma~\ref{l.pallot}.  
Therefore, using monotonicity, 
the uniform $p$-fatness condition~\eqref{e.fatness},
and the comparison inequality~\eqref{e.comparison}, we obtain
\begin{align*}
\cp_p(\overline{2^{-1}\widehat{B}}\cap \{v=0\},\widehat B)&\ge 
\cp_p(\overline{2^{-1}\widehat{B}}\cap E_{B},\widehat B)=\cp_p(\overline{2^{-1}\widehat{B}}\cap E,\widehat B)\\
& \ge c_0 \cp_p(\overline{2^{-1}\widehat{B}},\widehat{B})\ge
\frac{c_0\,\mu( 2^{-1}\widehat{B})}{C_{X,p} 2^{-mp}r^p}\,. 
\end{align*}
Finally, since $\tau \widehat{B}\subset \tau \widetilde{B}$, 
it follows that
\begin{align*}
C_p \lvert v_{\widehat{B}}\rvert^p
\le C_p\intav_{\widehat{B}}\lvert v(x)\rvert^p\,d\mu(x)
\le \frac{C_{X,p}}{c_0} 
\frac{2^{-mp} r^p}{\mu(2^{-1}\widehat{B})}\int_{\tau\widetilde{B}}g(x)^p\,d\mu(x)\,.
\end{align*}
Inequality~\eqref{e.single} follows from the above estimates 
and the doubling condition~\eqref{e.doubling}.
\end{proof}

\begin{proof}[Proof of Proposition \ref{p.Wannebo}] 
Let us first assume that $v\in \Lip_0(X\setminus E_{B})$ and that
$g_v$ is a $p$-weak upper gradient of $v$ 
that also vanishes in the set $E_{B}$. 
Then, by summing the inequalities~\eqref{e.single}  
and using~\eqref{e.finover} and~\eqref{e.ees}
we obtain, for every $m\in\N_0$,
\begin{equation}\label{e.coll1}
\begin{split}
\int_{G_m} \lvert v(x)\rvert^p\,d\mu(x)&\le \sum_{\widetilde{B}\in\mathcal{G}_m} \int_{\widetilde{B}} \lvert v(x)\rvert^p\,d\mu(x)
\\&\le \frac{C_{X,p}}{c_0} 2^{-mp} r^p\sum_{\widetilde{B}\in\mathcal{G}_m}\int_{\tau\widetilde{B}\setminus E_{B}}g_v(x)^p\,d\mu(x)
\\&\le \frac{C_{X,p}}{c_0} 2^{-mp} r^p\int_{\widetilde{G}_{m-\ell}} g_v(x)^p\,d\mu(x)\,.
\end{split}
\end{equation} 
Let $0<\beta<1$ be a small number, that will be fixed later. 
We multiply~\eqref{e.coll1} by $2^{m(p+\beta)}r^{-p-\beta}$ and sum the inequalities to
obtain the estimate
\begin{equation}\label{e.improvebeta}
\begin{split}
\int_{2B\setminus E_{B}} &\frac{\lvert v(x)\rvert^p }{d_{E_{B}}(x)^{p+\beta}}\,d\mu(x)\le  
  \int_{\widetilde{G}_0} \frac{\lvert v(x)\rvert^p }{d_{E_{B}}(x)^{p+\beta}}\,d\mu(x)=
  \sum_{m=0}^\infty \int_{G_m} \frac{\lvert v(x)\rvert^p }{d_{E_{B}}(x)^{p+\beta}}\,d\mu(x)
\\&\le \sum_{m=0}^\infty2^{m(p+\beta)}r^{-p-\beta} \int_{G_m} \lvert v(x)\rvert^p \,d\mu(x)\\
&\le \frac{C_{X,p}}{c_0}  r^{-\beta}\sum_{m=0}^\infty 2^{m\beta}\int_{\widetilde{G}_{m-\ell}} g_v(x)^p\,d\mu(x)\\
&=\frac{C_{X,p}}{c_0}  r^{-\beta}\sum_{k=-\ell}^\infty \sum_{m=0}^{k+\ell} 2^{m\beta}\int_{G_{k}} g_v(x)^p\,d\mu(x)\\
&\le \frac{C_{X,p}}{c_0 \beta}  \sum_{k=-\ell}^\infty 2^{k\beta}r^{-\beta}\int_{G_{k}} g_v(x)^p\,d\mu(x)
\le  \frac{C_{X,p}}{c_0 \beta}  \int_{8\tau B\setminus E_{B}} \frac{g_v(x)^p}{d_{E_{B}}(x)^\beta}\,d\mu(x)\,.
\end{split}
\end{equation} 

Now we come to the main line of the argument. 
Let $u$ be a Lipschitz function on $X$ and let $g$ be a
$p$-weak upper gradient of $u$, both of
which vanish in an open set $U\subset X$ satisfying the condition $\dist(E_{B},X\setminus U)>0$. 
We aim to show that inequality~\eqref{e.int} holds,
and so we may assume that $g\in L^p(8\tau B)$ (recall that $\tau \ge 1$).

Consider the Lipschitz function on $A=8\tau B \cup (X\setminus 10\tau B)$ that 
coincides with $u$ in $8\tau B$ and vanishes outside $10\tau B$,  
and let $\tilde u$ be the McShane extension \eqref{McShane} of this function to all of $X$.
Then the function 
\[\tilde g= g\chi_{8\tau B} + \Lip(\tilde u,\cdot)\chi_{X\setminus 8\tau B}\in L^p(X)\]
is a $p$-weak upper gradient of $\tilde u$, cf.\ the proof of \cite[Theorem~2.6]{MR2867756}.
Define $v(x)=\tilde u(x) d_{E_{B}}(x)^{\beta/p}$ for every $x\in X$.
Then $v(x)=u(x) d_{E_{B}}(x)^{\beta/p}$ for every $x\in 8\tau B$ and 
in particular $v$ vanishes in $E_{B}$. 
Moreover, by applying the assumptions on $u$ and $g$ in combination with 
the Leibniz and chain rules of
Theorems~2.15 and~2.16 in~\cite{MR2867756}, we find that $v$ has a 
$p$-weak upper gradient $g_v$ such that
\[
g_v(x) \le  g(x) d_{E_{B}}(x)^{\beta/p} + \frac{\beta}{p} \lvert u(x)\rvert d_{E_{B}}(x)^{\beta/p-1}
\]
for every $x\in 8\tau B$; in particular, also $g_v$ vanishes on the set $E_{B}$.
Using estimate~\eqref{e.improvebeta} for the pair $v$ and $g_v$, we obtain  
\begin{align*}
\int_{2B\setminus E_{B}} \frac{\lvert u(x)\rvert^p}{d_{E_{B}}(x)^p}\,d\mu(x)& 
  = \int_{2B\setminus E_{B}} \frac{\lvert v(x)\rvert^p}{d_{E_{B}}(x)^{p+\beta}}\,d\mu(x)
\le \frac{C_{X,p}}{c_0 \beta}  \int_{8\tau B\setminus E_{B}} \frac{g_v(x)^p}{d_{E_{B}}(x)^\beta}\,d\mu(x)\\
&\le \frac{C_{X,p}}{c_0 \beta}  \int_{8\tau B} g(x)^p\,d\mu(x) + \frac{C_{X,p}}{c_0} 
\beta^{p-1} \int_{8\tau B\setminus E_{B}} \frac{\lvert u(x)\rvert^p}{d_{E_{B}}(x)^p}\,d\mu(x)\,.
\end{align*} 
We can now apply Lemma~\ref{l.absorb} with parameters
\[
\varsigma=2\,,\qquad \sigma = 4\tau \,,\qquad q=p\,,\qquad C_1 = \frac{C_{X,p}}{c_0} \beta^{p-1}\,,\qquad 
C_2=\frac{C_{X,p}}{c_0\beta}\,.
\]
Recall our convention in \S\ref{s.constants} and choose $0<\beta<1$, depending on $C_X$, $p$, and $c_0$, 
such that
\[
C_3=1-C_1(1+C_{C_D,\sigma,\varsigma,p})\ge \frac{1}{2}\,.
\]
Then, Lemma~\ref{l.absorb} yields that
\begin{align*}
\int_{8 \tau B\setminus E_{B}}\frac{\lvert u(x)\rvert^p}{d_{E_{B}}(x)^p}\,d\mu(x)&
\le C_{X,p,c_0}\int_{8\tau^2 B} g(x)^p\,d\mu(x)\,,
\end{align*}
and this concludes the proof.
\end{proof}

Before entering the final stage in our proof of the self-improvement of uniform $p$-fatness,
we take a small side step and give a proof for
Theorem~\ref{t.Wannebo_intro} that was stated in the introduction.

\begin{proof}[Proof of Theorem~\ref{t.Wannebo_intro}]
Fix $w\in E$ and $0<r<(1/8)\diam(X)$, and 
let $E_{B}$ be as in Lemma~\ref{l.truncation}
for the ball $B=B(w,r)$.
Fix $u\in \Lip_0(X\setminus E)$ and a $p$-weak upper gradient $g$ of $u$. 
Since we first aim to prove estimate~\eqref{e.est} below, we can
assume that  $g\in L^p(8\tau^2 B)$. 

For every $\delta>0$, we define a Lipschitz function $u_\delta=\max\{0,\lvert u\rvert-\delta\}$.
Since $g$ is clearly a $p$-weak upper gradient of $u_\delta$, it is straightforward
to show that the function
\[
h=g\chi_{8\tau^2 B}+\Lip(u_\delta,\cdot)\chi_{X\setminus 8\tau^2 B}
\]
is a $p$-weak upper gradient of $u_\delta$, cf.\ the proof of \cite[Theorem~2.6]{MR2867756}.
Since $u_\delta$ vanishes in the set $U_\delta=\{\lvert u\rvert < \delta\}$ 
 we can apply the local version of the glueing lemma \cite[Lemma~2.19]{MR2867756}
with  $u_\delta$ and $h$. From this we can deduce
that $g_\delta=h\chi_{X\setminus U_\delta}$ is  a $p$-weak upper gradient of $u_\delta$.
Observe that both $u_\delta$ and $g_\delta$ vanish in the open
neighbourhood $U_\delta$ of $E$ and $\dist(E_B,X\setminus U_\delta)>0$ if $U_\delta \not=X$.
Since $E\cap (1/2)B\subset E_B$ and $w\in E_B$, we see that $d_E=d_{E_B}$ in $(1/4)B$. Hence,
 by monotone convergence and Proposition \ref{p.Wannebo}, we conclude that
\begin{equation}\label{e.est}
\begin{split}
\int_{(1/4)B\setminus E}\frac{\lvert u(x)\rvert^p}{d_{E}(x)^p}\,d\mu(x)
&=\lim_{\delta\to 0}\int_{(1/4)B\setminus E}\frac{\lvert u_\delta(x)\rvert^p}{d_{E_B}(x)^p}\,d\mu(x)\\
&\le C_H \liminf_{\delta\to 0} \int_{8\tau^2 B} g_\delta(x)^p\,d\mu(x)
\le C_H\int_{8\tau^2 B} g(x)^p\,d\mu(x)\,.
\end{split}
\end{equation}
The desired inequality~\eqref{e.int_intro} now follows by a simple change of variables.
\end{proof}

\subsection{Improvement}\label{s.improvement}

In this section we improve the `local integral Hardy inequality', that 
was established in  Proposition~\ref{p.Wannebo},  by adapting
ideas from Koskela--Zhong~\cite{MR1948106} to the present setting
and applying again the absorption Lemma~\ref{l.absorb}.
This improvement argument constitutes the final step in the proof of Theorem~\ref{t.main}.

\begin{proposition}\label{p.improved}
Let $1<p<\infty$ 
and suppose that $X$ supports the improved $(q,q)$-Poincar\'e inequalities (PI)
for $p_0\le q\le p$. 
Assume that $E\subset X$ is a uniformly $p$-fat closed set, let
$w\in E$ and $0<r< (1/8) \diam(X)$, and let $E_{B}$ be as in Lemma~\ref{l.truncation}
for $B=B(w,r)$.
Then  there exists constants $0<\varepsilon=\varepsilon_{X,p_0,p,C_H}<p-p_0$ and 
$C>0$ such that the inequality
\begin{equation}\label{e.improved}
\int_{12\tau^2 B\setminus E_{B}} \frac{\lvert u(x)\rvert^{p-\varepsilon}}{d_{E_{B}}(x)^{p-\varepsilon}}\,d\mu(x)
\le C \int_{12\tau^3 B}  \Lip(u,x)^{p-\varepsilon}\,d\mu(x)
\end{equation}
holds whenever $u\in \Lip_0(X\setminus E)$.
Here $C_H=C_{X,p,c_0}$ is the constant from Proposition~\ref{p.Wannebo}.
\end{proposition}

In the proof of Proposition~\ref{p.improved}, we use
the restricted maximal function $M_R u$ at $x\in X$,
for $R:X\to [0,\infty)$ and a locally integrable function $u$ on $X$,
that is defined by $M_R u(x)=\lvert u(x)\rvert$ if $R(x)=0$, and otherwise by
\[
M_R u(x) = \sup_r \intav_{B(x,r)}\lvert u(y)\rvert\,d\mu(y)\,,
\] 
where the supremum is taken over all radii $0<r<R(x)$.

\begin{proof}[Proof of Proposition \ref{p.improved}] 
Without loss of generality, we may assume that $C_H\ge 1$ in \eqref{e.int}.
We will first prove inequality \eqref{e.improved} 
under the additional assumption that $u\in \Lip(X)$ is such that 
$u=0$ in an open set $U\subsetneq X$ for which
$\dist(E,X\setminus U)>0$. 
Throughout this proof, we write $g=\Lip(u,\cdot)$; in particular, also $g=0$ in $U$.

Fix a number $\lambda>0$, and define $F_\lambda = H_\lambda \cap G_\lambda$, where
\[\begin{split}
H_\lambda & = \{x\in 8\tau^2 B\,:\, \lvert u(x)\rvert \le \lambda d_{E_{B}}(x)\}\,,\\
G_\lambda & = \big\{x\in 8\tau^2 B\,:\,  \big(  M_{d_{E_{B}}(x)/2}\,g^{p_0}(x)\big)^{1/p_0}\le \lambda\big\}\,.
\end{split}
\]
We claim that the restriction of $u$ to $F_\lambda$ is
$(C_{X,p_0}\lambda)$-Lipschitz. Indeed, let $x,y\in F_\lambda$, $x\not=y$,
be such that $d_{E_{B}}(y)\le d_{E_{B}}(x)$.
If $d_{E_{B}}(x)\ge 5\tau d(x,y)$, then
\[d_{E_{B}}(y)\ge d_{E_{B}}(x) - d(x,y) \ge 4\tau d(x,y)\,.\]
Thus a standard chaining argument \cite[Theorem~3.2]{MR1683160},
which is based on the facts that $\mu$ is doubling
and that the $(1,p_0)$-Poincar\'e inequality holds
for the pair $u$ and $g=\Lip(u,\cdot)$,
yields that
\begin{align*}
\lvert u(x)-u(y)\rvert &\le C_{X,p_0}\,d(x,y)\Big(  \big(  M_{2\tau d(x,y)}\,g^{p_0}(x)\big)^{1/p_0}
    + \big(  M_{2\tau d(x,y)}\,g^{p_0}(y)\big)^{1/p_0}\Big)\\
&\le C_{X,p_0}\,d(x,y)\Big(  \big(  M_{d_{E_{B}}(x)/2}\,g^{p_0}(x)\big)^{1/p_0}
    + \big(  M_{d_{E_{B}}(y)/2}\,g^{p_0}(y)\big)^{1/p_0}\Big)\\
&\le C_{X,p_0}\lambda d(x,y)\,. 
\end{align*}
On the other hand, if $d_{E_{B}}(y)\le d_{E_{B}}(x)\le 5\tau d(x,y)$, then
\[
\lvert u(x)-u(y)\rvert \le \lvert u(x)\rvert + \lvert u(y)\rvert \le \lambda (d_{E_{B}}(x) + d_{E_{B}}(y))
\le 10 \tau \lambda d(x,y).
\]
These two estimates show that $u$ is Lipschitz on $F_\lambda$.

Next we use the McShane extension \eqref{McShane} and extend the restriction of $u$ on $A=F_\lambda$ 
to a $(C_{X,p_0}\lambda)$-Lipschitz function $\tilde u$ on $X$.
Then also $\tilde u$ vanishes on an open set $\widetilde U\subset U$ such that 
$\dist(E_{B},X\setminus \widetilde U)>0$; 
indeed, if $x\in \widetilde U = \{x\in 8\tau^2 B\,:\,d_{E_{B}}(x)<\dist(E_{B},X\setminus U)/2\}$ 
then $u(x)=0$ and $x\in F_\lambda$ (here we use the fact that $g=\Lip(u,\cdot)=0$ in $U$), 
whence $\tilde u(x)=0$.

By~\cite[Lemma~2.19]{MR2867756},
the bounded function
\[
\tilde g(x) =  \chi_{F_\lambda}(x)g(x) + C_{X,p_0}\lambda \chi_{X\setminus F_\lambda}(x)
\]
is a $p$-weak upper gradient of $\tilde u$ 
that vanishes on $\widetilde U$. 
Hence, applying Proposition~\ref{p.Wannebo} to the pair
$\tilde u$ and $\tilde g$, we find that
\begin{align*}
\int_{(8\tau B\setminus E_{B})\cap F_\lambda} \frac{\lvert u(x)\rvert^p}{d_{E_{B}}(x)^p}\,d\mu(x)&
 \le \int_{8\tau B\setminus E_{B}} \frac{\lvert \tilde u(x)\rvert^p}{d_{E_{B}}(x)^p}\,d\mu(x)\\
&\le C_H\int_{F_\lambda} g(x)^p\,d\mu(x) +C_H C_{X,p_0}^p\lambda^p \mu(8\tau^2 B\setminus F_\lambda)
\end{align*}
and, since $C_H\ge 1$, that
\begin{align*}
\int_{(8\tau B\setminus E_{B})\cap H_\lambda} & \frac{\lvert u(x)\rvert^p}{d_{E_{B}}(x)^p}\,d\mu(x)\\
\le C_H & \int_{F_\lambda} g(x)^p\,d\mu(x)
  +C_H C_{X,p_0}^p \lambda^p \mu(8\tau^2 B\setminus F_\lambda)
  +\int_{(H_\lambda\setminus E_{B})\setminus G_\lambda}\frac{\lvert u(x)\rvert^p}{d_{E_{B}}(x)^p}\,d\mu(x)\\
\le C_H & \int_{G_\lambda} g(x)^p\,d\mu(x)
+C_H C_{X,p_0,p} \lambda^p \big(\mu( 8\tau^2 B\setminus F_\lambda)
+\mu( H_\lambda\setminus G_\lambda)\big)\\
\le  C_H & \int_{G_\lambda} g(x)^p\,d\mu(x)
+C_H C_{X,p_0,p} \lambda^p \big(\mu( 8\tau^2 B\setminus H_\lambda)
+\mu( 8\tau^2 B\setminus G_\lambda)\big)\,.
\end{align*}

The above estimate holds for all $\lambda>0$. We multiply it by $\lambda^{-1-\varepsilon}$ 
(here $0<\varepsilon<(p-p_0)/2$ is a parameter to be fixed later) and integrate the resulting 
estimate over $(0,\infty)$. This gives
\begin{align*}
\varepsilon^{-1} \int_{8\tau B\setminus E_{B}} & \frac{\lvert u(x)\rvert^{p-\varepsilon}}{d_{E_{B}}(x)^{p-\varepsilon}}\,d\mu(x)\\
\le C_H & \int_0^\infty  \lambda^{-1-\varepsilon}\int_{G_\lambda} g(x)^p\,d\mu(x)\,d\lambda\\
& + C_H C_{X,p_0,p}\int_0^\infty  \lambda^{p-1-\varepsilon} \big(\mu( 8\tau^2 B\setminus H_\lambda)
+\mu( 8\tau^2 B\setminus G_\lambda)\big)\,d\lambda\,.
\end{align*}
By the definition of $G_\lambda$, and the observation that
$g(x) \le \big(  M_{d_{E_{B}}(x)/2}\,g^{p_0}(x)\big)^{1/p_0}$ for a.e.\ $x\in 8\tau^2 B$, 
we find that the first term
on the right-hand side is dominated by
\[
C_H \varepsilon^{-1}\int_{8\tau^2 B} g(x)^{p-\varepsilon}\,d\mu(x)\,.
\]
Using the definitions
of $H_\lambda$ and $G_\lambda$, the second term on the right-hand side 
can be estimated from above by
\[
\frac{C_H C_{X,p_0,p}}{p-\varepsilon}\bigg(
\int_{8\tau^2 B\setminus E_{B}} \frac{\lvert u(x)\rvert^{p-\varepsilon}}{d_{E_{B}}(x)^{p-\varepsilon}}\,d\mu(x)
+ \int_{8\tau^2 B} \big(  M_{d_{E_{B}}(x)/2}\,g^{p_0}(x)\big)^{\frac{p-\varepsilon}{p_0}}\,d\mu(x)
\bigg)\,.
\]
Since $d_{E_{B}}(x)/2\le 4\tau^2 r$ for all $x\in 8\tau^2 B$, 
we have by the Hardy--Littlewood maximal theorem, see e.g.\ \cite[Theorem~3.13]{MR2867756}, that
\begin{align*}
\int_{8\tau^2 B} \big(  M_{d_{E_{B}}(x)/2}\,g^{p_0}(x)\big)^{\frac{p-\varepsilon}{p_0}}\,d\mu(x)
&\le \int_{X} \big(M(\chi_{12\tau^2 B}\,g^{p_0})(x)\big)^{\frac{p-\varepsilon}{p_0}}\,d\mu(x)\\
&\le C_{C_D,p_0,p,\varepsilon}\int_{12\tau^2B} g(x)^{p-\varepsilon}\,d\mu(x)\,;
\end{align*}
here $M$ denotes the usual unrestricted maximal operator.

By combining the estimates above, we obtain
\begin{align*}
\int_{8\tau B \setminus E_{B}} & \frac{\lvert u(x)\rvert^{p-\varepsilon}}{d_{E_{B}}(x)^{p-\varepsilon}}\,d\mu(x)\\
&\le C_1\int_{8\tau^2 B\setminus E_{B}} \frac{\lvert u(x)\rvert^{p-\varepsilon}}{d_{E_{B}}(x)^{p-\varepsilon}}\,d\mu(x) +  
     C_2\int_{12\tau^2 B} g(x)^{p-\varepsilon}\,d\mu(x)\,,
\end{align*}
where
\[
C_1 = C_H  C_{X,p_0,p} \varepsilon (p-\varepsilon)^{-1}\qquad\text{ and }\qquad
C_2 = C_H(1+C_{X,p_0,p} \varepsilon(p-\varepsilon)^{-1}  C_{C_D,p_0,p,\varepsilon})\,.
\] 
In order to apply Lemma~\ref{l.absorb}, we write
\[
\varsigma=8\tau\,,\qquad \sigma =3\tau/2 \,,\qquad q=p-\varepsilon\,.
\]
Recall our convention in \S\ref{s.constants} and choose $0<\varepsilon<(p-p_0)/2$, 
depending on $X$, $p_0$, $p$, and $C_H$, in such a way that
\[
C_3=1-C_1(1+ C_{C_D,\sigma,\varsigma,p}) \ge 1/2\,.
\] 
Then, Lemma~\ref{l.absorb} yields that 
\begin{align*}
\int_{12\tau^2 B\setminus E_{B}} \frac{\lvert u(x)\rvert^{p-\varepsilon}}{d_{E_{B}}(x)^{p-\varepsilon}}\,d\mu(x)
\le 2C_3\int_{12\tau^2 B\setminus E_{B}} \frac{\lvert u(x)\rvert^{p-\varepsilon}}{d_{E_{B}}(x)^{p-\varepsilon}}\,d\mu(x)
\le 2C_4\int_{12\tau^3 B} g(x)^{p-\varepsilon}\,d\mu(x)\,,
\end{align*}
where $C_4=(1+C_2)C_{X,\sigma,\varsigma,p-\varepsilon}$. 
This proves the claim in the case where $u=0$ in an open set $U\subset X$ with
$\dist(E,X\setminus U)>0$.

To prove the general case, let $u\in \Lip_0(X\setminus E)$. Clearly, we may
assume that $u$ does not vanish everywhere  in $X$.
For every $\delta>0$, we define a Lipschitz function $u_\delta=\max\{0,\lvert u\rvert-\delta\}$.  
Now  $\Lip(u_\delta,\cdot)\le \Lip(u,\cdot)$ and $u_\delta$ vanishes in the open
neighbourhood $U_\delta=\{\lvert u\rvert < \delta\}$ of $E$.
Thus, by monotone convergence and
the special case of inequality~\eqref{e.improved} that was established above
we conclude that
\begin{align*}
\int_{12\tau^2 B\setminus E_{B}} \frac{\lvert u(x)\rvert^{p-\varepsilon}}{d_{E_{B}}(x)^{p-\varepsilon}}\,d\mu(x)
&=\lim_{\delta\to 0} \int_{12\tau^2 B\setminus E_{B}} 
    \frac{\lvert u_\delta(x)\rvert^{p-\varepsilon}}{d_{E_{B}}(x)^{p-\varepsilon}}\,d\mu(x)\\
&\le C \liminf_{\delta\to 0} \int_{12\tau^3 B}  \Lip(u_\delta,x)^{p-\varepsilon}\,d\mu(x)\\
&\le C \int_{12\tau^3 B}  \Lip(u,x)^{p-\varepsilon}\,d\mu(x)\,.
\end{align*}
This proofs the claim in the general case $u\in \Lip_0(X\setminus E)$.
\end{proof}

\bibliographystyle{abbrv}
\def\cprime{$'$}

\end{document}